\documentclass[12pt]{amsart}  
\usepackage{amssymb}  
\usepackage{amsthm}
\usepackage{amsmath}
\usepackage{graphicx}
\usepackage{mathabx}
\usepackage{relsize}
\usepackage{pinlabel}
\newtheorem{Thm}{Theorem}

\newtheorem{Conj}[Thm]{Conjecture}

\theoremstyle{definition}

\theoremstyle{definition}
\newtheorem{Rem}[Thm]{Remark}
\newtheorem{Ex}[Thm]{Example}

\theoremstyle{definition}

\title{Band Number and the Double Slice Genus}
\author{Clayton McDonald}
\address{Department of Mathematics, Boston College,140 Commonwealth Avenue, Chestnut Hill, MA 02467
}
\email[Clayton McDonald]{mcdonafi@bc.edu}

\begin{document}

\begin{abstract}
We study the double slice genus of a knot, a natural generalization of slice genus. We define a notion called band number, a natural generalization of band unknotting number, and prove it is an upper bound on double slice genus. Our bound is based on an analysis of broken surface diagrams and embedding properties of 3-manifolds in 4-manifolds.  
\end{abstract}

\maketitle


The study of surfaces in $B^4$ that a knot bounds is a well developed theme in low dimensional topology. Rather than studying properly embedded surfaces in $B^4$ which a given knot $K \subset S^3 = \partial B^4$ bounds, one can instead imagine two properly embedded surfaces $(S_1,K) \hookrightarrow (B_1^4,S^3)$ and $(S_2,K) \hookrightarrow (B_2^4,S^3)$, and then glue these two 4-balls together to produce an embedded surface $F \subset S^4$ whose intersection with the meridional $S^3$ is $K$. 

The \textbf{double slice genus} of $K$, $g_{ds}(K)$, is defined as the minimal genus of such an $F$ such that $F$ is unknotted, i.e. bounds a handlebody. This notion of unknottedness is a natural extension of the 1-knot case, where one definition of the unknot is a knot that bounds an embedded disc. We note that a handlebody is in some sense the simplest 3-manifold bounded by a surface, and that any two surfaces that bound handlebodies of the same genus are isotopic: the existence of a bounding handlebody, which retracts to a wedge of circles in $S^4$, guides the isotopy between the surfaces.


Double slice genus naturally lends itself to the study of embedded 3-manifolds in 4-manifolds through branched coverings (compare \cite{agl,donald15}). More specifically, if a knot $K$ on a meridional $S^3 \subset S^4$ lies on a particular knotted surface $F \subset S^4$, then by taking the $n$-fold cyclic cover of $S^4$ branched along $F$, we get a 4-manifold $\Sigma_n(S^4,F)$ with a map to $S^4$ that is 1-to-1 on $F$, and $n$-to-1 otherwise. The preimage of the meridional $S^3$ is a 3-manifold with a map to $S^3$ that is 1-to-1 along $F \cap S^3 = K$ and $n$-to-1 otherwise. Therefore, this manifold is $\Sigma_n(S^3,K)$, so we have an embedding of $\Sigma_n(S^3,K)$ into $\Sigma_n(S^4,F)$.

It follows that $g_{ds}(K)$ is a natural upper bound on $\varepsilon(\Sigma_2(S^3,K))$, the minimal $n$ for which the branched double cover of $K$ embeds in $\#_n S^2 \times S^2$ \cite[Definition 2.2]{agl}. This is because the branched double cover of $S^4$ with branch set an unknotted genus $g$ surface is $\#_g S^2 \times S^2$, so the branched double cover of $K$ embeds in $\#_g S^2 \times S^2$ for $g = g_{ds}$. There are other bounds on $\varepsilon(K)$ coming from more classical knot invariants, the Seifert genus $g_3(K)$ and unknotting number $u(K)$. Both of these bounds can be seen using the Montesinos trick on a presentation of the knot, obtaining an even integral surgery description of the branched double cover. The doubles of the traces of these integral surgery descriptions yield connect sums of $S^2 \times S^2$, giving $2g_3$ and $2u$ as upper bounds for $\varepsilon$.

The quantity $2g_3(K)$ can also easily be shown to be an upper bound for $g_{ds}(K)$, as the double of a Seifert surface for a knot is unknotted. This is because pushing the two copies of the Seifert surface into the two 4-balls sweeps out a handlebody which the resulting surface bounds. The quantity $2u(K)$ is also an upper bound for $g_{ds}(K)$, but this fact is less obvious and seen more easily from the results of this paper, which relate these invariants to band unknotting number.


An (oriented) \textbf{band unknotting sequence} for $K$ is a sequence of oriented saddle moves on $K$ that yields an unknot at the end (Figure \ref{fig:saddle}). If we follow this process in reverse from the unknot to $K$, each oriented saddle move corresponds to an immersed band attachment to the disc that the unknot bounds. Therefore, an oriented band unknotting sequence of length $2N$ yields a ribbon immersed surface $S_0$ with one disc and $2N$ bands, which we may arrange such that the bands are pairwise disjoint and intersect the disc only in ribbon singularities. Furthermore, we can promote this ribbon immersed surface to a ribbon surface $S$, a properly embedded surface in $B^4$ with only 0 and 1-handles. We do this by pushing the interior of the disc into $B^4$, pushing the disc portions of the ribbon singularities further into $B^4$ to remove the intersections. We define $u_b(K)$, the (oriented) \textbf{band unknotting number}, as the minimum length of a band unknotting sequence for $K$. 
\begin{figure}
    \centering
    \includegraphics[width=0.5\textwidth]{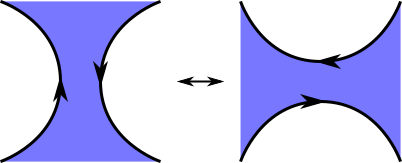}
    \caption{An oriented saddle move.}
    \label{fig:saddle}
\end{figure}

This band unknotting number can be seen as a unification of unknotting number and Seifert genus in this context, as it is a lower bound for both. Any crossing change can be obtained from two oriented saddle moves, and the existence of a Seifert surface gives a band unknotting sequence from the handle decomposition of the surface with length twice its genus. Furthermore, it is an upper bound on double slice genus and even the superslice genus  $g_{ss}$, defined below:
\begin{Thm}
\label{thm: ub}
For a knot K in $S^3$,
\begin{equation}
  u_b(K) (\geq g_{ss}(K)) \geq g_{ds}(K).
\end{equation}
\end{Thm}

To prove Theorem \ref{thm: ub}, we use the double of the ribbon surface coming from the band unknotting sequence. The surface has a handle decomposition with one 0-handle and $u_b(K)$ 1-handles, so its double is a surface of genus $u_b(K)$ in $S^4$. Thus we will proceed by proving that this doubled surface is unknotted. We therefore establish that $u_b(K)$ is an upper bound on superslice genus $g_{ss}(K)$, which is defined as the minimal genus of an unknotted surface in $S^4$ that arises as the double of a slice surface of $K$. It is clear that $g_{ss}(K)$ an upper bound for $g_{ds}(K)$, as it is a more restrictive condition on the construction of such an unknotted surface.

Our proof will involve manipulations of broken surface diagrams for 2-knots \cite{cs}. A broken surface diagram for a 2-knot is exactly analogous to a knot diagram for a 1-knot. For a knot diagram, we can remove two balls from $S^3$ to get $S^2 \times [-1,1]$, and then project to the $S^2$ coordinate to depict our knot as an immersed circle in $S^2$, along with extra crossing information, which indicates which portion of the knot is higher in the $[-1,1]$ coordinate at self intersections. Similarly, for an embedding of a surface into $S^4$, we remove two 4-balls, project the surface from $S^3 \times [-1,1]$ to $S^3$, and obtain an immersed surface in $S^3$, along with some ``crossing information" for the self intersections, i.e. which part of the surface is higher in the $[-1,1]$ coordinate.

The resulting \textbf{broken surface diagram} captures the isotopy type of the embedded surface in $S^4$. As with knot diagrams, an arbitrary projection can be very badly behaved, but after a slight perturbation, we can guarantee an immersion with regularity properties for the singular set, such as transverse self intersections. 

We now describe a procedure to get from a ribbon surface $S$ for a knot to a broken surface diagram of its double in $S^4$. As stated before, we can remove a small neighborhood from each $B^4$ on each side of our meridional $S^3$, and think of the resulting $S^3 \times [-1,1]$  as our ambient space, as this admits a natural projection $\pi: S^3 \times [-1,1] \mapsto S^3$. We will call the $[-1,1]$ factor the $w$ coordinate, and the other factor the $S^3$ coordinate. We start with two copies of $S_0$, the ribbon immersed surface corresponding to $S$, in the meridional $S^3 \times {0}$. Take one copy of $S_0$ and push its interior into the bottom $B^4$ (i.e. down in $w$). Then push the other copy into the top $B^4$ (up in $w$) to form two embedded ribbon surfaces in their respective punctured 4-balls.
We will call these $S_-$ and $S_+$ respectively; their union $F = S_+ \cup S_-$ is the double of $S$.
Note that $S_+$ and $S_-$ both project to $S_0$ under $\pi$.
Although $S_0$ is immersed, it is still oriented, and thus two-sided. Therefore, this gives us two distinct directions in the $S^3$ coordinate in which we can push the interior of $S_0$.
By pushing $S_-$ in one of those directions and $S_+$ in the other direction, we obtain a regular immersion of $F$ with transverse self intersections by projecting it to the meridional $S^3$, which we now show how to diagrammatize.

\begin{figure}
    \centering
    \includegraphics[width=0.6\textwidth]{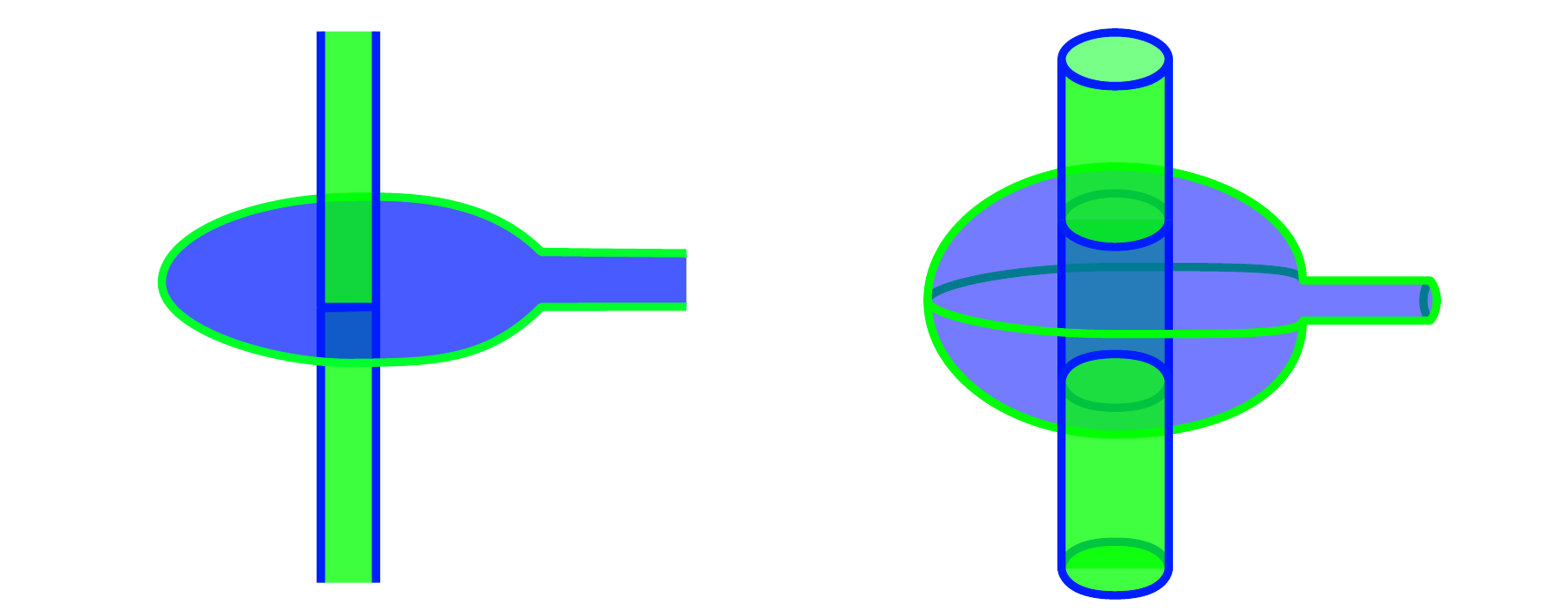}
    \caption{A ribbon immersed surface $S_0$ (left) and its double (right).}
    \label{fig:double}
\end{figure}

When we perform this pushoff (seen in Figure \ref{fig:double}), each disc in our disc-band presentation of $S$ gives rise to a sphere, and each band gives rise to a tube connecting two patches of these spheres. Thus we obtain a sphere-tube presentation of the double of $S$. Each ribbon singularity with a disc will form a circle as we inflate our bands into tubes, so when we push the two copies of the disc off each other, we will have two such intersection circles for each ribbon singularity of $S$. These pushoffs come with crossing information inherited from the projections of $S_+$ and $S_-$ to $S^3$.
Note that we have a natural hierarchy on the $w$ coordinate because of he way we pushed the two copies of $S_0$ to form the resulting embedded surface in $S^4$. The discs of $S_+$ are above its bands, which are in turn above all of $S_-$. In $S_-$, the opposite is true, as we pushed everything in the negative direction. This means that the tube part of the surface has a higher $w$ coordinate when it passes through $\pi(S_-)$, and a lower $w$ coordinate when it passes through $\pi(S_+)$.

Next we can discard our choices for a $w$ function on the surface, noting that any choice that preserves the crossing information of the broken surface will preserve the isotopy type (much like for a knot diagram). 
\begin{proof}[Proof of Theorem \ref{thm: ub}]
Let $S$ be a one-disk ribbon surface.
We claim that the double of $S$ is unknotted.
Note that if $S$ is obtained from a band unknotting sequence, then the genus of the double equals the number of bands used, so the theorem is reduced to proving the claim.
We proceed to establish the claim by induction on the number of ribbon singularities in $S_0$, the immersion of $S$.
In the base case where $S_0$ has no ribbon singularities, $S_0$ is in fact a Seifert surface, and the unknottedness of its double is clear.
Now suppose that $S_0$ has at least one ribbon singularity and the result is true for any one-disk ribbon surface whose immersion has fewer ribbon singularities.
The aim of the proof will be to use a few specific moves on the double of $S$ to pass to a ribbon surface whose immersion has one fewer ribbon singularity and whose double is isotopic to that of $S$.

Consider a band $B$ of $S_0$ containing a ribbon singularity and consider the first ribbon singularity of this band starting from one foot of the band.
We explain how to isotope $S$ to remove this ribbon singularity from $S_0$ without introducing any new ones. 

There are two arcs in $S_0$ from the foot of the band to the ribbon singularity, one along the core of $B$ and another along the disc.
We identify a Whitney disk $D$ in $S^3$ whose boundary is composed of the union of these two arcs which will guide the isotopy cancelling the ribbon singularity. We may assume that the interior of $D$ is disjoint from the disc of $S_0$, and intersects bands other than $B$ transversely in ribbon intersections. We can then perform a sequence of band crossing changes between $B$ and any bands intersecting $D$ as in \cite[Section 4.2]{ml} and Figure \ref{bcs} to remove all of the ribbon intersections with $D$.
Note that this move does not introduce any new ribbon singularities, but can change the isotopy type of the boundary knot. However, it does not change the isotopy type of the doubled surface. This is because one can push one of the two tubes higher than the other in the $w$ coordinate, so that when they pass through each other in the $S^3$ coordinate, they do not intersect in $S^3 \times [-1,1]$.

\begin{figure}[h]

    \labellist
    \small\hair 2pt
    \pinlabel $D$ at 250 450
    \endlabellist
    \centering
    \includegraphics[width=0.4\textwidth]{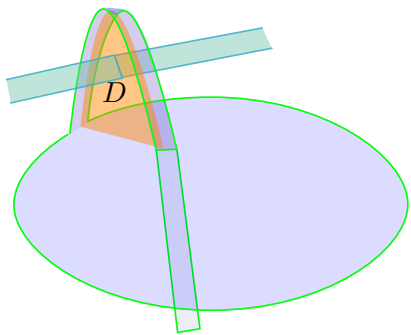}
    \caption{The cancelling Whitney disc $D$.}
    \label{fig:whitney}
\end{figure}
\begin{figure}[h]

    \centering
    \includegraphics[width = 0.5\textwidth]{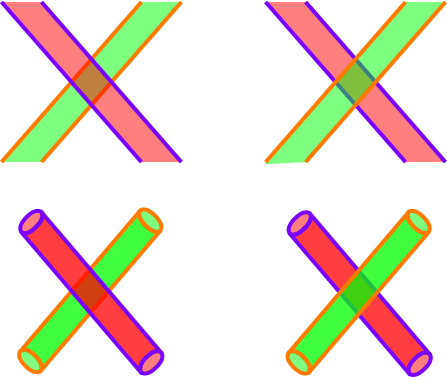}
    \caption{The isotopy of the doubled surface (below) inducing the band crossing change (above).}
    \label{bcs}
\end{figure}

Then, as there are no ribbon singularities with $D$, we can perform a Whitney move along $D$ cancelling the ribbon singularity. We now have a ribbon surface whose double is isotopic to that of $S$, but has one fewer ribbon singularity. Therefore by induction, its double, and thus the double of $S$, is unknotted.
\end{proof}
Using similar methods as in Theorem \ref{thm: ub}, we can prove a slightly stronger bound on $g_{ds}$ using tubings, i.e. 2-D 1-handles, in the broken surface diagram.
\begin{Thm}
\label{thm: b}
Given a ribbon surface $S$ for a knot $K$ in $S^3$ with $d$ discs and $b$ bands, $b$ is an upper bound for $g_{ds}(K)$.
\end{Thm}
This means that the minimal number of bands among all ribbon surfaces with boundary $K$, which we will denote as the \textbf{band number} $b(K)$, is an upper bound on double slice genus. The quantity $b(K)$ is a generalization of $u_b(K)$, where it can be seen as the distance via oriented saddle moves between $K$ and some unlink, instead of the unknot in the case of $u_b(K)$. It is also a generalization of fusion number $f(K)$, or ribbon fusion number of a ribbon knot, discussed in \cite{kanenobu}, the minimum number of bands in a ribbon disc for $K$. Not only does $b(K)$ extend the definition to non-ribbon knots, it also is theoretically lower, as the minimum number of bands is not a priori realized by a minimum genus ribbon surface.

\begin{proof}
First, use the same doubling procedure as before to go from a ribbon surface with $d$ discs and $b$ bands to a broken surface diagram of the double with $d$ spheres and $b$ tubes between them, such that the only self intersections in the projection are double circles between the tubes and the spheres. This is a broken surface diagram of a potentially knotted surface in $S^4$ such that there is some meridional $S^3$ whose intersection with this surface is $K$.

Then tube together all of the spheres with trivial tubes (one can imagine the core of such a tube as an arc between the two spheres that does not intersect the immersed surface) to get a surface with one sphere and $b$ tubes. Note that if we assume that the attaching regions of the tubes we added are both on the bottom disc of the sphere, i.e. the one with  $w < 0$, then it becomes clear that we can do this tubing such that there is still a meridional $S^3$ that intersects the resulting surface in $K$, as in Figure \ref{fig:tube}. 

\begin{figure}[h]
    \labellist
    \small\hair 2pt
    \pinlabel $S^4$ at 150 530
    \pinlabel $F$ at 320 430
    \pinlabel tube at 650 340
    \pinlabel $S^3$ at 410 550
    \endlabellist
    \centering
    \includegraphics[width=0.5\textwidth]{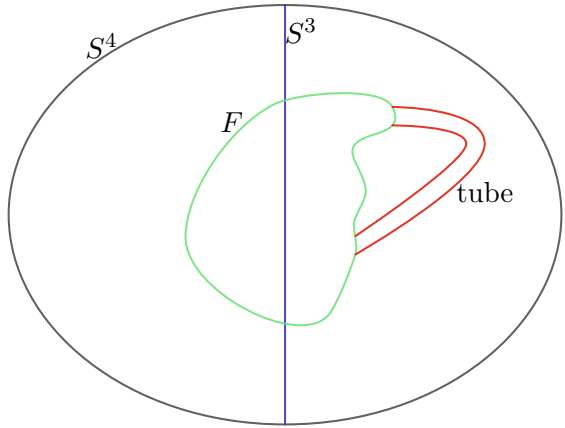}
    \caption{A schematic for tube attachment to $F$ as to avoid the meridional $S^3$.}
    \label{fig:tube}
\end{figure}

From here we note that because the resulting surface has a one sphere, $b$ tube presentation, Theorem \ref{thm: ub} shows that it is an unknotted genus $b$ surface, and thus the double slice genus of $K$ is at most $b$.
\end{proof}

Note that because the tubings we use are not symmetric, $b$ is not a bound on $g_{ss}$.
Unlike the other bounds for $g_{ds}$, $b$ can take odd values. In particular, we can use this to prove that the stevedore knot $6_1$ has $g_{ds} = 1$, as it has a two-disc, one-band ribbon disc, but it is not doubly slice because the first homology of its branched double cover is not a direct double, as in \cite[Proposition 2.1]{ml}.

\begin{Rem}We can improve on this bound in certain situations. Using the band crossing change argument from before, we see that the isotopy class of the 2-knot gotten by doubling a ribbon disc depends only on the homotopy classes of the cores of the bands in $X$, the complement of the boundaries of the discs in $S^3$.
(In fact, the isotopy class only depends on the bands' homotopy classes up to band swim and handle slide equivalence).
There is an action of $\pi_1(X)$ on these homotopy classes, where a group element acts by concatenating that loop onto our arc, using a fixed set of arcs from the basepoint of $X$ to the boundary components of $X$.

The fundamental group $\pi_1(X)$ is a free group generated by meridians of the disc boundaries, and the group action acts transitively on the set of homotopy classes beginning and ending on a given pair of boundary components of $X$, as concatenating a meridian to an existing path allows an arc to pass through the corresponding disc boundary. Therefore, we can encode, with some redundancy, the homotopy classes of these arcs by their start and endpoints as well as an element of $\pi_1(X)$.
Moreover, we note that the word we designate in $\pi_1(X)$ corresponds to the orderings of the ribbon singularities of the corresponding band, as a meridian of a disc boundary only intersects its disc in a single point. This means that the ordering of the ribbon singularities with sign encapsulates all of the information about the homotopy class of the band.

The core of the argument in Theorem \ref{thm: ub} was that we could cancel the closest ribbon singularity to the foot of the band if it is with the disc that foot is attached to.
Additionally, although we attached the tubes in Theorem \ref{thm: b} such that they didn't intersect the meridional $S^3$, they are isotopic to doubles of trivial bands between discs.
Adding such a band has the effect on $\pi_1(X)$ of identifying the meridians of the two discs that the band connects.
Moreover, the double of our ribbon surface after trivial banding between the discs is still a surface which contains $K$ as a cross section. Therefore, in our tubing process if we ever make a partial set of identifications of generators such that all of these homotopy classes of bands can be cancelled, then this would be a more refined bound on double slice genus.
\end{Rem}
\begin{Ex}

Consider a four-disc ($D_1$, $D_2$, $D_3$, $D_4$) three-band ($B_1$ from $D_1$ to $D_2$, $B_2$ from $D_2$ to $D_3$, $B_3$ from $D_3$ to $D_4$) ribbon disc with $B_1$ passing through $D_3$, $B_2$ passing through $D_1$, and $B_3$ passing through $D_4$ and then $D_2$.
Theorem \ref{thm: b} gives an upper bound of three on the double slice genus of the boundary knot.
However, we see that if we trivially band $D_1$ to $D_3$ and $D_2$ to $D_4$, $B_1, B_2$ and $B_3$ are homotopically trivial in the resulting complement of disc boundaries, so we get a bound of two.
\end{Ex}

We close with some conjectures that illustrate the limits of our knowledge:
As indicated above, the band number of a ribbon knot is a lower bound on its ribbon fusion number.
We conjecture that the two are not always equal, i.e that there exists a ribbon knot that bounds a ribbon surface with $b$ bands, yet any ribbon disc for it has more than $b$ bands. 

For example, the Whitehead double of any knot has Seifert genus one, and therefore has band number at most two. Moreover, the untwisted Whitehead pattern is a saddle move away from the 0-framed 2-cable pattern. Therefore we can construct a ribbon disc for the double of a ribbon knot by doing this saddle move and then appending two ribbon discs for the 2-cable. The natural band presentation of the ribbon disc for the untwisted Whitehead double has $2b+1$ bands if the original ribbon disc had $b$ bands, making these Whitehead doubles natural candidates for this conjecture. However, untwisted Whitehead doubles of ribbon knots are all superslice \cite{gordonsumners}, so they have band presentations with homotopically trivial bands. Therefore, many of the properties of these knots will be hard to detect algebraically.

More generally, we expect a similar statement is true for every genus: \begin{Conj}
For every genus $g$, there exists a knot $K$ with ribbon genus $g$ for which $b(K)$ is lower than the number of bands in any ribbon surface of genus $g$.
\end{Conj}

Demonstrating this difference seems challenging, as the lower bounds we know how to prove for ribbon fusion number give lower bounds for band number as well.

\section*{Acknowledgements.} Thanks to Antonio Alfieri, for piquing my interest in this subject, as well as to my advisor Joshua Greene and Maggie Miller for helpful conversations.

\bibliographystyle{plain}
\bibliography{main}

\end{document}